\documentclass[12pt]{article}
\title{The Knotting-Unknotting Game played on Sums of Rational Shadows}
\author{Will Johnson}
\usepackage{amsfonts}
\usepackage{amsmath}
\usepackage{amsmath, amssymb, amsthm}    	% need for subequations
\usepackage{fullpage} 	% without this, will have wide math-article-paper margins
\usepackage{graphicx}	% use to include graphics
\usepackage{float}
\usepackage{verbatim}
\usepackage{hyperref}

\newcommand{\outcome}{\operatorname{o}}

\newtheorem{theorem}{Theorem}[section] % numbered like the section
\newtheorem{lemma}[theorem]{Lemma} % numbered like the theorems

\newtheorem{corollary}[theorem]{Corollary}
\newtheorem{definition}[theorem]{Definition}

\newtheorem{remark}[theorem]{Remark}

\begin{document}
\maketitle

\section{Introduction}
A \emph{knot pseudodiagram} is a knot projection with some crossing data left unspecified, like Figure~\ref{example-pseudodiagram}.
\begin{figure}[htb]
\begin{center}
\includegraphics[width=3in]
					{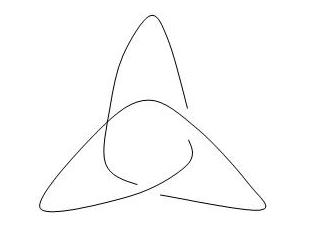}
\caption{A knot pseudodiagram.  One of the crossings is an \emph{unresolved crossing} (precrossing), while the other two are resolved.}
\label{example-pseudodiagram}
\end{center}
\end{figure}
In a pseudodiagram, crossings are allowed to be \emph{unresolved}, meaning that they do not indicate which strand is on top.
A pseudodiagram in which every crossing is unresolved is called a \emph{knot shadow}.  Knot shadows and pseudodiagrams were invented
by Ryo Hanaki~\cite{Hanaki}, motivated by the problem of mathematically modeling microscopic images of DNA with unclear crossing information.ling microscopic images of DNA with unclear crossing information.
See \cite{Hanaki} and \cite{Janos} for more information on pseudodiagrams.

The \emph{knotting-unknotting game}, also known as \emph{To Knot or Not to Knot}, is the game played on a knot
shadow or pseudodiagram as follows:
Two players, the \emph{Knotter} and \emph{Unknotter} (also known as \emph{King Lear} and \emph{Ursula}), take turns \emph{resolving} unresolved crossings, as
in Figure~\ref{fig:resolution},
until the knot is fully determined. Then the Unknotter wins if the resulting diagram is equivalent to the unknot, and the Knotter wins otherwise.
This game was introduced in \cite{KnotGames}.
\begin{figure}[htb]
\begin{center}
\includegraphics[width=3in]
					{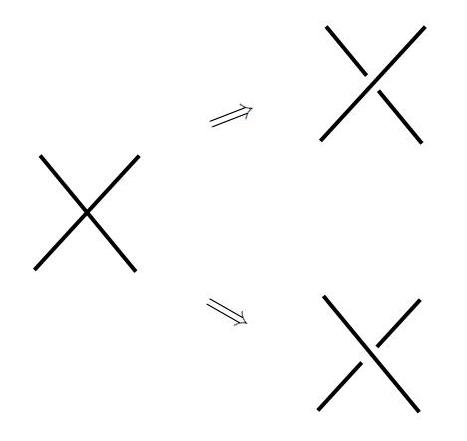}
\caption{The two ways to resolve an unresolved crossing.}
\label{fig:resolution}
\end{center}
\end{figure}

The practical difficulty with this game is determining which player has won once the game is over.  There is no simple rule to test whether a knot is
equivalent to the unknot.  However, if we restrict the game to \emph{rational shadows}, like the one in Figure~\ref{rational-shadow},
then the final knot will be a \emph{rational knot}.
In this case, a relatively simple rule due to John Conway~\cite{Conway1970} determines whether the knot is equivalent to the unknot.  This can be generalized slightly
to the case of \emph{sums of rational shadows}, like the one in Figure~\ref{rational-shadow-sum}.
\begin{figure}[htb]
\begin{center}
\includegraphics[width=3in]
					{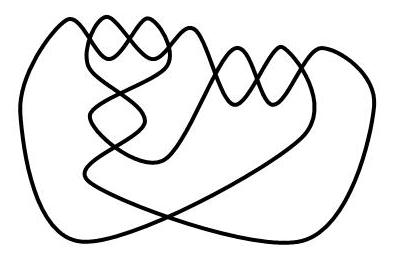}
\caption{A rational shadow, analogous to a rational knot in the sense of Conway.}
\label{rational-shadow}
\end{center}
\end{figure}
\begin{figure}[htb]
\begin{center}
\includegraphics[width=3in]
					{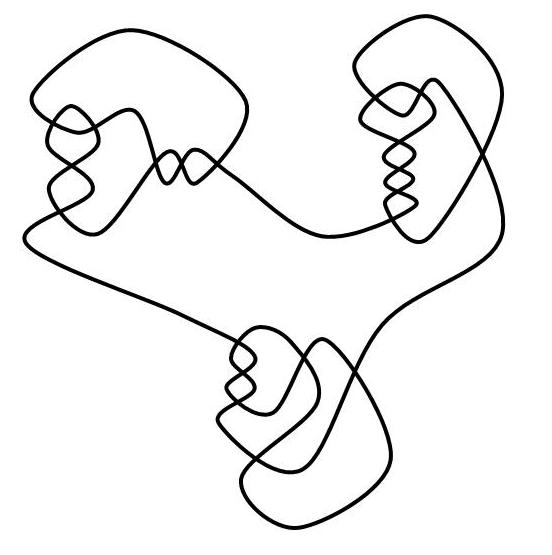}
\caption{A sum of rational shadows.}
\label{rational-shadow-sum}
\end{center}
\end{figure}

In this paper, we determine which player wins in the knotting-unknotting game, for all starting positions which are sums of rational shadows.
First, we define pseudo Reidemeister I and II moves (see Figures~\ref{phonyr1} and \ref{phonyr2}) in Section~\ref{sec:ops}, and consider their general
strategic effects in
the knotting-unknotting game in Section~\ref{sec:outcomes}.  We then
turn in Section~\ref{sec:rational} to the case of rational shadows, defining them and showing which operations correspond to pseudo Reidemeister moves.
In Sections~\ref{sec:oddeven} and \ref{sec:remainder} we show that the Unknotter has a guaranteed win for a certain small family of rational shadows, while in all
other cases, the winner is the second or first player, depending on the parity of the number of crossings.
This requires a computer verification of a finite list
of minimal cases; see appendix~\ref{sec:py} for the relevant python code.  In Section~\ref{sec:sum} we consider sums of rational shadows, and determine the winner
in all such positions.

\section{Operations on Pseudodiagrams}\label{sec:ops}
If $G$ and $H$ are two pseudodiagrams, we can define the \emph{(connected) sum} of $G$ and $H$, denoted $G\#H$, in a way completely analogous to the
usual definition for knots.  For example, the connected sum of the pseudodiagrams in Figure~\ref{trefoilstar-and-ambiguous} is shown
in Figure~\ref{trefoilstar-plus-ambiguous}.
\begin{figure}[htb]
\begin{center}
\includegraphics[width=3in]
					{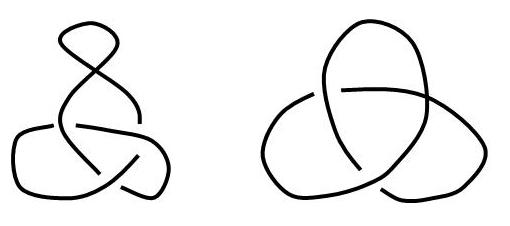}
\caption{Two pseudodiagrams}
\label{trefoilstar-and-ambiguous}
\end{center}
\end{figure}
\begin{figure}[htb]
\begin{center}
\includegraphics[width=2in]
					{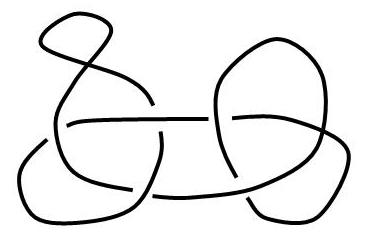}
\caption{The connected sum of the two pseudodiagrams in Figure~\ref{trefoilstar-and-ambiguous}}
\label{trefoilstar-plus-ambiguous}
\end{center}
\end{figure}
There are several ambiguities in this definition.  However, we will consider
pseudodiagrams to be equivalent if they can be related by the moves of Figure~\ref{modulo-moves}, which have no strategic effects.  Modulo these moves,
$G\#H$ is unambiguous.  This operation is associative and commutative.
\begin{figure}[htb]
\begin{center}
\includegraphics[width=5in]
					{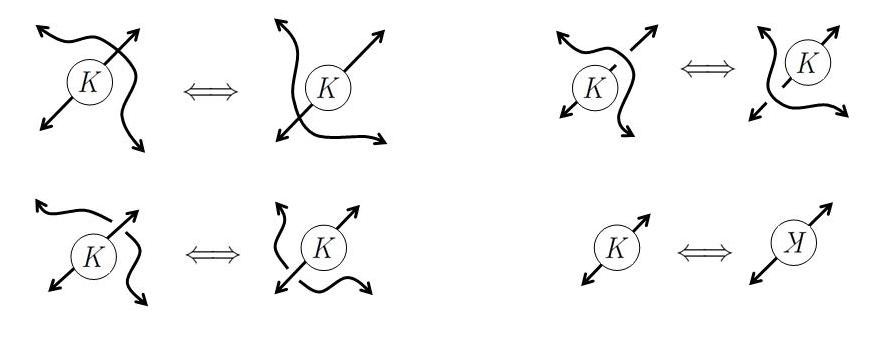}
\caption{These moves have no strategic effect on a pseudodiagram, from the point of view of the knotting-unknotting game.  Here $K$ stands
for an arbitrary pseudodiagram.  We consider two pseudodiagrams equivalent if they can be related by these moves or by (non-pseudo)
Reidemeister moves.}
\label{modulo-moves}
\end{center}
\end{figure}

By considering the genus of a knot, one can show that the connected sum of two knots is the unknot if and only if
the two knots are both the unknot.  For example this is done on pages 99-104 of Adams~\cite{KnotBook}.
From the point of view of the knotting-unknotting game, this means that when playing the sum of two positions,
the Unknotter needs to win on both summands separately to win the sum game.  The Knotter, on the other hand, only needs to win
on one of the two summands.  This makes the operation of adding knots inherently asymmetric, biased towards
the Knotter.

If $P$ and $Q$ are pseudodiagrams, we use $P \stackrel{1}{\to} Q$ to indicate that $Q$ is obtained from $P$ by deleting an unresolved loop, as in Figure~\ref{phonyr1}.
We call this a \emph{pseudo Reidemeister I move}.
\begin{figure}[htb]
\begin{center}
\includegraphics[width=3in]
					{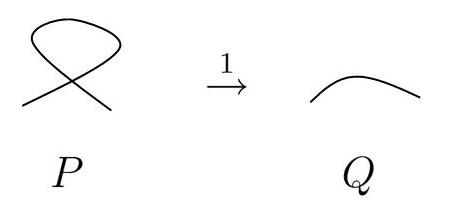}
\caption{Pseudo Reidemeister I move}
\label{phonyr1}
\end{center}
\end{figure}
Note that $P \stackrel{1}{\to} Q$ if and only if $P$ is $Q\#*$, where $*$ is the pseudodiagram of Figure~\ref{star-game}.  (The symbol
$*$ is motivated by analogy with the game $*$ in combinatorial game theory.)

\begin{figure}[htb]
\begin{center}
\includegraphics[width=1.3in]
					{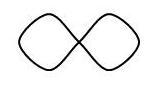}
\caption{The game $*$.  There is one available move, and the game is a guaranteed win for the Unknotter, regardless
of how play proceeds.}
\label{star-game}
\end{center}
\end{figure}

Similarly, we use $P \stackrel{2}{\to} Q$ to indicate that $Q$ is obtained from $P$ by a pseudo-Reidemeister II move, as in Figure~\ref{phonyr2}.
\begin{figure}[htb]
\begin{center}
\includegraphics[width=3in]
					{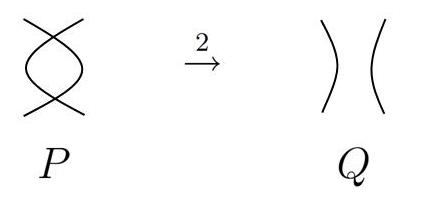}
\caption{Pseudo Reidemeister II move}
\label{phonyr2}
\end{center}
\end{figure}
We also use the notation $P \stackrel{i}{\Rightarrow} Q$ to indicate that $Q$ is obtained from $P$ by a sequence
of zero or more moves of type $i$, and $P \stackrel{*}{\Rightarrow} Q$ to allow for a mixture of both types of moves.
(So $\stackrel{*}{\Rightarrow}$ is the smallest reflexive and transitive relation containing both $\stackrel{1}{\to}$ and
$\stackrel{2}{\to}$.)  We also say that $P$ \emph{reduces to} $Q$ if $P \stackrel{*}{\Rightarrow} Q$.

We do not consider pseudodiagrams to be equivalent if they can be related by pseudo Reidemeister moves, because these operations
can change the outcome of the game.  The effect of pseudo Reidemeister moves on outcomes will be the focus of the next
section.

\begin{remark}\label{rem}
If $\Box$ is any of $\stackrel{1}{\to}$, $\stackrel{2}{\to}$, $\stackrel{1}{\Rightarrow}$,
$\stackrel{2}{\Rightarrow}$, $\stackrel{*}{\Rightarrow}$, then $G \Box H$ implies that $G\#K \Box H\#K$.
\end{remark}

\section{Outcomes}\label{sec:outcomes}

The knotting-unknotting game is a two-player finite game of perfect information with no ties or draws.  As such,
one of the two players has a winning strategy.  The identity of this player depends on which player goes first.
Consequently, we can group positions into four \emph{outcome classes}:
\begin{itemize}
\item Knotter wins under perfect play, no matter who goes first.
\item Unknotter wins under perfect play, no matter who goes first.
\item Whoever goes first wins under perfect play.
\item Whoever goes second wins under perfect play.
\end{itemize}
We refer to these four possibilities as K, U, 1, and 2, respectively.  We also say that a position is K1 if it is K or 1, U2 if it is U or 2, and so on.
Note that each position is either K1 or U2, and either K2 or U1.  Also, a position is in U1 iff Unknotter can win playing 1st,
K1 iff Knotter can win playing first, and so on.  The various possibilities are illustrated in Figure~\ref{outcomes-diagram}.

We can also think of U2 and K2 as the positions which the Unknotter or Knotter (respectively) can safely move to,
while K1 and U1 are the positions that the Knotter or Unknotter (respectively) would like to receive from his or her opponent.

\begin{figure}[htb]
\begin{center}
\includegraphics[width=3in]
					{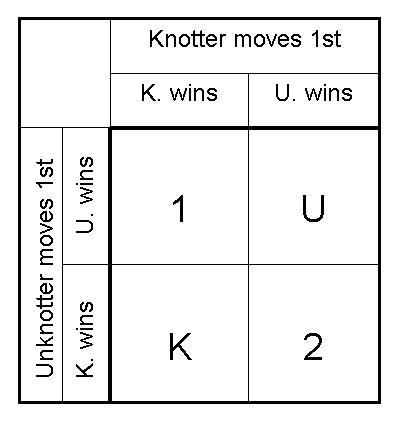}
\caption{The four possible outcome classes.}
\label{outcomes-diagram}
\end{center}
\end{figure}

%TODO: decide whether to use the preposition ``in'' with U2 like stuff.

\begin{definition}
If $P$ is a pseudodiagram, we say that $P$ is \emph{even} or \emph{odd} if the number of unresolved crossings
is even or odd, respectively.  We let the \emph{parity} of $P$, denoted $\pi(P)$, be $0$ or $1$,
if $P$ is even or odd, respectively.
\end{definition}
Note that if we play the knotting-unknotting game on $P$, the length of the game is the number of unresolved
crossings, so the parity of $P$ is the parity of the length of the game played on $P$.  Also note that if $P\stackrel{1}{\to} Q$,
then $P$ and $Q$ have opposite parities, while if $P \stackrel{2}{\to} Q$, then $P$ and $Q$ have the same parity.
Moreover, for any $P$ and $Q$, $\pi(P\#Q)$ has the same parity as $\pi(P) + \pi(Q)$.

\begin{definition}
If $P$ is a pseudodiagram, then an \emph{option} of $P$ is a pseudodiagram $Q$ obtained by resolving one crossing.
\end{definition}
Note that the options of $P$ all have the opposite parity to $P$.
If $P$ has no options, then every crossing in $P$ is resolved, so $P$ is in fact a true knot diagram.  We say that
$P$ is \emph{fully resolved} in this case.

The outcome classes
listed above can be given opaque recursive definitions as follows:
\begin{itemize}
\item If $P$ is a knot diagram of the unknot, then $P$ is U2 and U1.
\item If $P$ is a knot diagram that is not the unknot, then $P$ is K2 and K1.
\item Otherwise,
\begin{itemize}
	\item $P$ is U2 iff all of its options are U1.
	\item $P$ is K2 iff all of its options are K1.
	\item $P$ is U1 iff at least one of its options is U2.
	\item $P$ is K1 iff at least one of its options is K2.
\end{itemize}
\end{itemize}
Then we define the four outcome classes themselves as follows:
\begin{itemize}
\item $P$ has outcome U iff it is U1 and U2
\item $P$ has outcome K iff it is K1 and K2
\item $P$ has outcome 1 iff it is U1 and K1
\item $P$ has outcome 2 iff it is U2 and K2
\end{itemize}

\begin{definition}
Let $P$ be a pseudodiagram.  Then we say that $P$ is a \emph{zero game} iff $P$ is U1 and for every option
$Q$ of $P$, there is an option $R$ of $Q$, such that $Q$ is a zero game.
\end{definition}

Since this definition is recursive, we can make inductive proofs:
\begin{lemma}\label{zlem}
If $P$ is a zero game, then $P$ is even and U2.
\end{lemma}
\begin{proof}
We proceed by induction.  First suppose that $P$ is odd.  Then $P$ has an odd number of unresolved crossings, and therefore at least one option $Q$.
By definition of zero game,
$Q$ has an option $R$ which is a zero game.  By induction, $R$ is even, so $Q$ is odd, and $P$ is even, a contradiction.

Next we show that $P$ is U2.  If $P$ is fully resolved, then $P$ is the unknot or not.  But by definition of zero game, $P$ is
U1, so it is the unknot and therefore also U2.  Otherwise, if $P$ is not fully resolved, then the Unknotter can reply to any
move from $P$ to $Q$ by moving from $Q$ to $R$, where $R$ is a zero game.  This is possible by definition of a zero game, and a
winning move by induction, which ensures that $R$ is a safe position for the Unknotter to move to.
\end{proof}
Since zero games are already U1, it follows that every zero game has outcome U.
Intuitively, a position is a zero game if both of the following are true:
\begin{itemize}
\item It is even
\item The Unknotter can win playing second, even if the Knotter is allowed on one of his turns to pass rather than play.
\end{itemize}

The next theorem shows that zero games are \emph{strategically trivial} in some sense:
\begin{theorem}\label{zerogames}
Let $A$ and $P$ be pseudodiagrams, with $P$ a zero game.  Then $A\#P$ has the same outcome as $A$.
\end{theorem}
\begin{proof}
We proceed by joint induction on the number of unresolved crossings in $A$ and $P$.
First suppose that $A$ is fully resolved.  Then $A$ has outcome class U or K, depending on whether
$A$ is the unknot or not.  If $A$ is the unknot, then $A\#P$ is the same as $P$, so by Lemma~\ref{zlem}, $P$ has outcome U.
Otherwise, $A$ is knotted.  Consequently, no matter what $P$ becomes, $A\#P$ will also end up becoming knotted,
so the Knotter is already guaranteed a win.  Then $A\#P$ has outcome K.  Either way, $A\#P$ has the same outcome as $A$.

Now suppose that $A$ is not fully resolved.  If it is U1, then some option $A'$ of $A$ is U2.  By induction, $A'\#P$ is U2.
But $A'\#P$ is an option of $A\#P$, so $A\#P$ is also U1.  In other words, if $A'$ is a good move for the Unknotter in $A$, then
$A'\#P$ is a good move for the Unknotter in $A\#P$.

Conversely, suppose that $A\#P$ is U1.  Then the Unknotter has some good move, either of the form $A'\#P$ or $A\#P'$.  In the first
case, $A'\#P$ is U2, so by induction, $A'$ is U2.  Therefore, $A$ is U1.  In the other case, $A\#P'$ is U2,
and it has some option of the form $A\#P''$ with $P''$ a zero game, because $P'$ is an option of a zero game.  But since $A\#P'$
is U2, $A\#P''$ must be U1, so by induction, $A$ is U1.

So we have just seen that $A\#P$ is U1 iff $A$ is U1.  Similar arguments show that $A\#P$ is K1 iff $A$ is K1.
Since a pseudodiagram's outcome class is determined by whether it is U1 and whether it is K1, it follows that $A$ and $A\#P$ have the same outcome.
\end{proof}
Informally, we could summarize this proof as follows: $A\#P$ has the same outcome as $A$, because a player with a winning strategy in $A$ can simply use the
same strategy in $A\#P$, responding to any move in the summand $P$ with a reply that reverts it to a zero game, and ensuring that $P$
turns into the unknot once $A$ becomes fully resolved.

Using this we see that two pseudo Reidemeister I moves have no effect on strategy:
\begin{corollary}\label{r1}
For any pseudodiagram $P$, $P$ and $P\#*\#*$ have the same outcome.
\end{corollary}
\begin{proof}
This follows from Theorem~\ref{zerogames} by showing that $*\#*$ is a zero game.  This is easy to check, however, since $*\#*$ is a
guaranteed win for the Unknotter no matter how the players play, and $*\#*$ is even.
\end{proof}

The effect of a single pseudo Reidemeister I move is more vague:
\begin{lemma}\label{r1bland}
If $P$ is U2, then $P\#*$ is U1.  Similarly, if $P$ is K2, then $P\#*$ is K1.
\end{lemma}
\begin{proof}
If the Unknotter can win in $P$ as the second player, then she can win in $P\#*$ as the first player
by moving in $*$ and then playing as the second player.  The sole option of $*$ is the unknot, so the Unknotter's opening move
results in a position equivalent to $P$.  The same trick works for the Knotter.
\end{proof}
Conversely, we also have
\begin{lemma}\label{r1bland-variant}
If $P\#*$ is U2, then $P$ is U1, and if $P\#*$ is K2, then $P$ is K1.
\end{lemma}
\begin{proof}
These statements are the logical contrapositives to Lemma~\ref{r1bland}.
\end{proof}

The situation with pseudo Reidemeister II moves is more complicated:
\begin{lemma}\label{r2}
Suppose $P \stackrel{2}{\to} Q$.  (In particular then, $P$ and $Q$ have the same
parity.)  If $P$ and $Q$ are even, then $Q$ is U2 implies $P$ is U2, and $Q$ is K2 implies $P$ is K2.
Similarly, if $P$ and $Q$ are odd, then $Q$ is U1 implies $P$ is U1, and $Q$ is K1 implies $P$ is K1.
\end{lemma}
Another way to say this, is to say that $P$ is no worse than $Q$, for \emph{the player who will make the last move of the game}.
In the even case, this is the second player, while in the odd case this is the first player.
\begin{proof}
Let ``Alice'' be the player who will make the last move of the game, and suppose that Alice has a winning strategy in $Q$.
Then she can use her strategy in $Q$ to win in
$P$.  If at any point her opponent moves in one of the two new crossings, she moves in the other in a way that makes a Reidemeister II move possible,
as in Figure~\ref{r2-responses}.
\begin{figure}[htb]
\begin{center}
\includegraphics[width=3in]
					{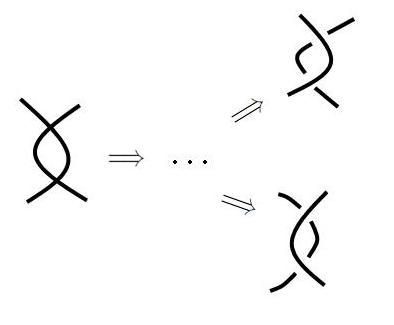}
\caption{After the first move, it is always possible to reply with a move to one of the configurations on the right.}
\label{r2-responses}
\end{center}
\end{figure}

Otherwise, she never plays in one of the two new crossings.  She is never forced to play in one of the two new crossings, because
this would only happen in the case where the two new crossings were the sole remaining places to move.  But if this were the case,
then on Alice's turn, there would be two more moves remaining in the game, so Alice's opponent would be the player who made the final move,
contradicting the choice of ``Alice.''
\end{proof}

We use a somewhat complicated method to describe the outcome of $P$ and $P\#*$ for a general pseudodiagram $P$.
\begin{definition}
If $P$ is a pseudodiagram, let $P^0$ and $P^1$ be the pseudodiagrams for which the unordered pairs $\{P^0,P^1\}$ and $\{P,P\#*\}$ are equal, $P^0$ is even, and $P^1$ is odd.
We call $P^0$
and $P^1$ the \emph{even} and \emph{odd projections} of $P$, respectively.
\end{definition}
Note that by Remark~\ref{rem}, if $P \stackrel{2}{\to} Q$, then $P^0 \stackrel{2}{\to} Q^0$
and $P^1 \stackrel{2}{\to} Q^1$.
\begin{definition}
If $P$ is a pseudodiagram, we let $\outcome(P) \in \{1,2,K,U\}$ denote the outcome class of $P$.
Then we call $(\outcome(P),\outcome(P\#*))$ the \emph{extended outcome} of $P$,
and $(\outcome(P^0),\outcome(P^1))$ the \emph{normalized outcome} of $P$.
\end{definition}
Note that if $P$ is even, then $(P^0,P^1) = (P,P\#*)$, so the extended and normalized outcome
are the same.  But if $P$ is odd, then $(P^0,P^1) = (P\#*,P)$, so the normalized outcome
is obtained from the extended outcome by swapping its components.  Moreover,
$P = P^{\pi(P)}$, so $\pi(P)$ and the normalized outcome of $P$ determine the outcome of $P$.

The point of normalized outcomes is the following:
\begin{lemma}\label{r1-normalized}
The normalized outcomes of $P$ and $P\#*$ are the same.  In particular, pseudo Reidemeister I
moves have no effect on normalized outcomes.
\end{lemma}
\begin{proof}
It is easy to see that $(P\#*)^0$ is either $P^0$ (if $P$ is odd), or $P^0\#*\#*$ (if $P$ is even).
So using Corollary~\ref{r1} if necessary, we see that $\outcome((P\#*)^0) = \outcome(P^0)$.  Similarly,
$(P\#*)^1$ is either $P^1$ (if $P$ is even), or $P^1\#*\#*$ (if $P$ is odd).
So using Corollary~\ref{r1} if necessary, we see that $\outcome((P\#*)^1) = \outcome(P^1)$.
\end{proof}

Lemmas~\ref{r1bland} and \ref{r1bland-variant} also impose some constraints on the possible normal outcomes.
In particular, for any $P$, we have $P^0 = P^1\#*$ or $P^1 = P^0\#*$, so by Lemma~\ref{r1bland} or Lemma~\ref{r1bland-variant},
the following possibilities are impossible:
\begin{itemize}
\item $P^0$ is U2 and $P^1$ is K2.
\item $P^0$ is K2 and $P^1$ is U2.
\end{itemize}
This makes the following definition legitimate
\begin{definition}\label{xydef}
Let $P$ be a pseudodiagram.  We then let $X(P)$ be determined as follows:
\begin{itemize}
\item $X(P) = 1$ iff $P^0$ is U2 and $P^1$ is U1.
\item $X(P) = 2$ iff $P^0$ is K1 and $P^1$ is U1.
\item $X(P) = 3$ iff $P^0$ is K1 and $P^1$ is K2.
\end{itemize}
Similarly, we define $Y(P)$ as follows:
\begin{itemize}
\item $Y(P) = 1$ iff $P^0$ is U1 and $P^1$ is U2.
\item $Y(P) = 2$ iff $P^0$ is U1 and $P^1$ is K1.
\item $Y(P) = 3$ iff $P^0$ is K2 and $P^1$ is K1.
\end{itemize}
\end{definition}
Note that higher values of $X(P)$ and $Y(P)$ are better for the Knotter
and lower values are better for the Unknotter.  Also note that the values
of $X(P)$ and $Y(P)$ together carry the exact same information as
the normalized outcome of $P$, since they determine for each $i \in \{0,1\}$
whether $P^i$ is in U2, and whether it is in K2, and these four possibilities
determine the outcomes of $P^i$, as in Figure~\ref{outcomes-diagram}.  The nine possibilities
are summarized in Figure~\ref{diamond}.  Note for instance that some possibilities,
like $(U,K)$, do not occur for normalized or extended outcomes.

The following rules are clear from the definition of $X(P)$ and $Y(P)$,
together with the fact that $P^{\pi(P)} = P$.
\begin{itemize}
\item If $P$ is even, then $P$ is U1 iff $Y(P) < 3$, while $P$ is K1 iff $X(P) > 1$.
\item If $P$ is odd, then $P$ is U1 iff $X(P) < 3$, while $P$ is K1 iff $Y(P) > 1$.
\end{itemize}

\begin{figure}[htb]
\begin{center}
\includegraphics[width=4in]
					{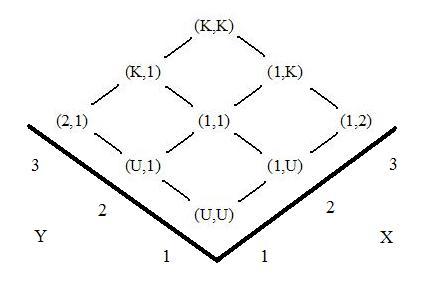}
\caption{Possible normalized outcomes, and their correspondence with $X$ and $Y$ values.  For example,
$X = 2$ and $Y = 3$ corresponds to a normalized outcome of $(K,1)$.}
\label{diamond}
\end{center}
\end{figure}
%Here we arranged them into a diamond, putting the better outcomes for the Knotter on top and the better outcomes for the Unknotter on the bottom.

This twisted way of describing the outcome of $P$ and $P\#*$ is motivated by the following
theorem:\footnote{Additionally, when considering sums of games, $X$ and $Y$ become relevant.  In particular,
$X(P\#Q)$ is completely determined by $X(P)$ and $X(Q)$, and $Y(P\#Q)$ is partially determined
by $Y(P)$ and $Y(Q)$, independently of the parities of $P$ and $Q$.  We do not discuss these facts in what follows,
though they should not be difficult for the interested reader to find.}
\begin{theorem}\label{xyinequalities}
If $P \stackrel{1}{\to} Q$, then
\[ X(P) = X(Q)\]
\[ Y(P) = Y(Q)\]
\[ \pi(P) = 1 - \pi(Q)\]
If $P \stackrel{2}{\to} Q$, then
\[ X(P) \le X(Q)\]
\[ Y(P) \ge Y(Q)\]
\[ \pi(P) = \pi(Q)\]
Consequently, if $P \stackrel{*}{\Rightarrow} Q$, then
\[ X(P) \le X(Q)\]
\[ Y(P) \ge Y(Q)\]
\end{theorem}
\begin{proof}
If $P \stackrel{1}{\to} Q$, then $P = Q\#*$.  By Lemma~\ref{r1-normalized}, $P$ and $Q$ have the same normalized outcomes, so they
have the same $X$ and $Y$ values.  On the other hand, they have opposite parities.

If $P \stackrel{2}{\to} Q$, then by Remark~\ref{rem}, $P^0 \stackrel{2}{\to} Q^0$ and $P^1 \stackrel{2}{\to} Q^1$.
Now $P^0$ and $Q^0$ are even, while $P^1$ and $Q^1$ are odd.  So by Lemma~\ref{r2},
\begin{itemize}
\item If $Q^0$ is U2, then $P^0$ is U2.
\item If $Q^0$ is K2, then $P^0$ is K2.
\item If $Q^1$ is U1, then $P^1$ is U1.
\item If $Q^1$ is K1, then $P^1$ is K1.
\end{itemize}
By Definition~\ref{xydef}, these amount to the following implications:
\[ X(Q) \le 1 \implies X(P) \le 1\]
\[ X(Q) < 3 \implies X(P) < 3\]
\[ Y(Q) \ge 3 \implies Y(P) \ge 3\]
\[ Y(Q) > 1 \implies Y(P) > 1.\]
But since $X$ and $Y$ values are in the set $\{1,2,3\}$, it follows easily that $X(P) \le X(Q)$ and
$Y(P) \ge Y(Q)$.

The inequalities for the $P \stackrel{*}{\Rightarrow} Q$ case follow by transitivity.
\end{proof}

\begin{corollary}\label{redux}
Let $P$ be a pseudodiagram that reduces to the unknot by pseudo Reidemeister I and II moves.  (That is,
$P \stackrel{*}{\Rightarrow} I$, where $I$ is the unknot.)  If $P$ is even, then its outcome class is either U or 2,
while if $P$ is odd, its outcome class is either U or 1.  In particular, no pseudodiagram that reduces to the unknot
is in class $K$.
\end{corollary}
\begin{proof}
The normalized outcome of the unknot is $(U,U)$, which corresponds to $X$ and $Y$ values of 1.  So by Theorem~\ref{xyinequalities},
$X(P) \le 1$ and $Y(P) \ge 1$.  This tells us nothing about $Y(P)$, but it tells us that $X(P) = 1$.  In particular,
$P^0$ is U2 and $P^1$ is U1.  But if $P$ is even, then $P^0 = P$, so $P$ is U or 2.  On the other hand, if $P$ is odd, then $P^1 = P$, so $P$
is U or 1.
\end{proof}
The simplest shadow that does not reduce to the unknot is the one shown in Figure~\ref{four-petal}.
\begin{figure}[htb]
\begin{center}
\includegraphics[width=1.5in]
					{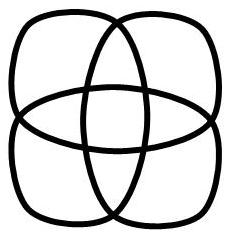}
\caption{The simplest knot shadow which does not reduce by pseudo Reidemeister I and II moves to the unknot.}
\label{four-petal}
\end{center}
\end{figure}
It can be verified with a computer computation that it indeed has outcome K.  (In fact the Knotter can guarantee that the final knot's
``knot determinant'' differs from the unknot.  The ``knot determinant'' is the magnitude of the
Alexander polynomial evaluated at $-1$.)

For a simpler example, the \emph{pseudodiagram} on the left side of Figure~\ref{trefoilstar-and-ambiguous} does not reduce to the unknot, and indeed it is in class K (trivially).

We also have
\begin{corollary}\label{21preserve}
If $P \stackrel{*}{\Rightarrow} Q$, and $Q$ has normalized outcome $(2,1)$, then so does $P$.
\end{corollary}
\begin{proof}
The normalized outcome of $Q$ is $(2,1)$ if and only if $X(Q) = 1$ and $Y(Q) = 3$, and similarly for $P$.
By Theorem~\ref{xyinequalities}, $X(P) \le X(Q)$ and $Y(P) \ge Y(Q)$.  So also $X(P) = 1$ and $Y(P) = 3$,
so $P$ has normalized outcome $(2,1)$.
\end{proof}

\section{Rational pseudodiagrams and shadows}\label{sec:rational}
The main class of pseudodiagrams that we consider are ones analogous to the rational knots in the sense of
Conway~\cite{Conway1970}.  These are formed recursively as follows:

First of all, we use $[~]$ to denote the rational tangle of Figure~\ref{base-tangle}.
\begin{figure}[htb]
\begin{center}
\includegraphics[width=1in]
					{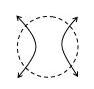}
\caption{The starting tangle}
\label{base-tangle}
\end{center}
\end{figure}
Then we recursively define $[a_1,\ldots,a_n]$ to be
the rational tangle obtained from $[a_1,\ldots,a_{n-1}]$ by reflecting over a 45 degree axis and adding $a_n$ twists
to the right.  A negative number indicates twists in the negative direction.  This notation is a variant of Conway's~\cite{Conway1970} notation for rational knots.
We also generalize this notation, letting $[a_1(b_1),\ldots,a_{n}(b_n)]$ denote a tangle pseudodiagram
in which there are $a_1$ legitimate crossings and $b_1$ unresolved crossings at each step.\footnote{The order of the resolved and unresolved crossings
is left unspecified because it makes no strategic difference.}  See Figure~\ref{tangle-examples}
for examples.

\begin{figure}[htbp]
\begin{center}
\includegraphics[width=6in]
					{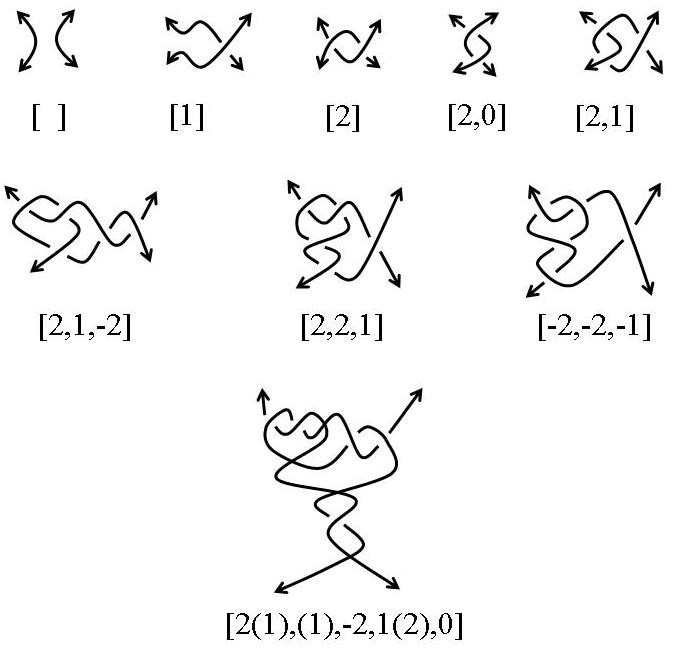}
\caption{Examples of rational tangle pseudodiagrams.}
\label{tangle-examples}
\end{center}
\end{figure}

In particular the $a_i \in \mathbb{Z}$ and the $b_i \in \mathbb{N}$, where $\mathbb{N}$ is the set of nonnegative integers.  If $a_i = 0$,
we write $(b_i)$ instead of $a_i(b_i)$, and similarly if $b_i = 0$, we write $a_i$ instead of $a_i(b_i)$.
A rational tangle shadow is one of the form $[(b_1),\ldots,(b_n)]$, in which no crossings are resolved.

We abuse notation, and use the same $[a_1(b_1),\ldots,a_n(b_n)]$ notation for the pseudodiagram obtained by
connecting the top two strands of the tangle and the bottom two strands, as in Figure~\ref{closing-a-tangle}.
\begin{figure}[htb]
\begin{center}
\includegraphics[width=3in]
					{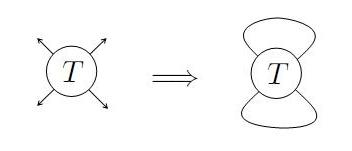}
\caption{A tangle is turned into a knot by connecting the top two strands, and connecting the bottom two strands.}
\label{closing-a-tangle}
\end{center}
\end{figure}
Note that this can sometimes yield a two-component link, rather than a knot, as in Figure~\ref{degenerate-shadow}.
\begin{figure}[htb]
\begin{center}
\includegraphics[width=2in]
					{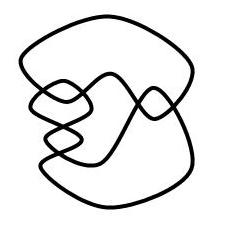}
\caption{Sometimes a tangle's closure is a two-component link, rather than a knot shadow.}
\label{degenerate-shadow}
\end{center}
\end{figure}

We list some fundamental facts about rational tangles, due to Conway~\cite{Conway1970}:
\begin{theorem}\label{fundrat}
If $[a_1,\ldots,a_m]$ and $[b_1,\ldots,b_n]$ are rational tangles, then they are equivalent if and only if
\[ a_m + \frac{1}{a_{m-1} + \frac{1}{\ddots + \frac{1}{a_1}}} = b_n + \frac{1}{b_{n-1} + \frac{1}{\ddots + \frac{1}{b_1}}}.\]
The link $[a_1,\ldots,a_m]$ is has a single component (i.e., is a knot) if and only if
\[ a_m + \frac{1}{a_{m-1} + \frac{1}{\ddots + \frac{1}{a_1}}} = \frac{p}{q},\]
where $p,q \in \mathbb{Z}$ and $p$ is odd.
Finally, $[a_1,\ldots,a_m]$ is the unknot if and only if $q/p$ is an integer.
\end{theorem}
Note that $[a_1(b_1),\ldots,a_n(b_n)]$ is a knot pseudodiagram (as opposed to a link pseudodiagram) if and only if
$[a_1 + b_1, \ldots, a_n + b_n]$ is a knot (as opposed to a link), since the number of components in the diagram
does not depend on how crossings are resolved.

\begin{lemma}
The following pairs of rational shadows are topologically equivalent (i.e., equivalent up to planar isotopy):
\begin{equation} \left[(1),(a_1),\ldots,(a_n)\right] = \left[(a_1 + 1),(a_2)\ldots,(a_n)\right] \label{onecombl}\end{equation}
\begin{equation} \left[(a_1),\ldots,(a_n),(1)\right] = \left[(a_1),\ldots,(a_{n-1}),(a_n + 1)\right] \label{onecombr}\end{equation}
\begin{equation} \left[(0),(0),(a_1),\ldots,(a_n)\right] = \left[(a_1),\ldots,(a_n)\right] \label{zlossl}\end{equation}
\begin{equation} \left[(a_1),\ldots,(a_i),(0),(a_{i+1}),\ldots,(a_n)\right] =
\left[(a_1),\ldots,(a_i + a_{i+1}),\ldots,(a_n)\right]\label{zloss}\end{equation}
\begin{equation} \left[(a_1),\ldots,(a_n),(0),(0)\right] = \left[(a_1),\ldots,(a_n)\right]\label{zlossr}\end{equation}
\begin{equation} \left[(a_1),(a_2),\ldots,(a_n)\right] = \left[(a_n),\ldots,(a_2),(a_1)\right] \label{invert}\end{equation}
\end{lemma}
\begin{figure}[htb]
\begin{center}
\includegraphics[width=4in]
					{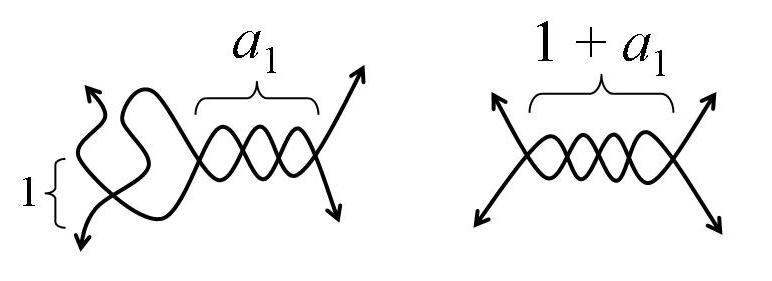}
\caption{Equation (\ref{onecombl}) holds because adding one twist to the bottom of Figure~\ref{base-tangle} and then $a_1$ twists on the right
is the same as adding $1+a_1$ twists to a sideways version of Figure~\ref{base-tangle}.}
\label{fig:onecombl}
\end{center}
\end{figure}
\begin{figure}[htb]
\begin{center}
\includegraphics[width=4in]
					{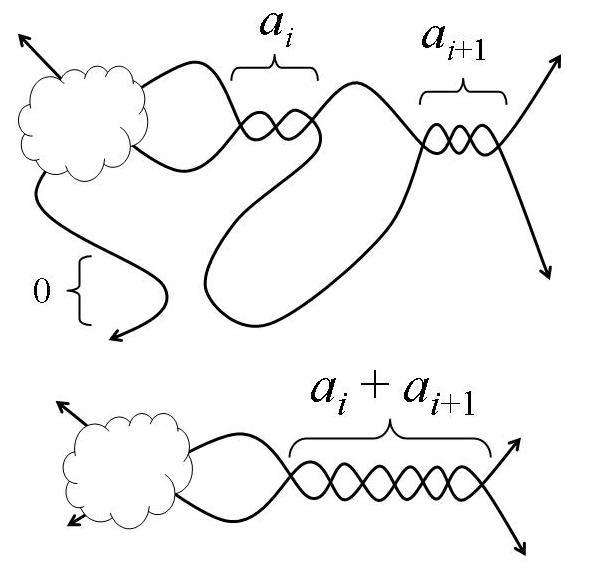}
\caption{Equation (\ref{zloss}) holds because adding $a_i$ twists on the right, $0$ twists on the bottom, and $a_{i+1}$ twists on the right of a tangle
is equivalent to simply adding $a_i + a_{i+1}$ twists on the right.}
\label{fig:zloss}
\end{center}
\end{figure}
\begin{proof}
Most of these can be easily seen by drawing pictures.  For example, (\ref{onecombl}) and (\ref{zloss}) follow by Figures~\ref{fig:onecombl}
and \ref{fig:zloss}, while (\ref{zlossl}) follows because twice reflecting Figure~\ref{base-tangle} has no effect.  The only non-obvious
equivalence is (\ref{invert}).  The equivalence here follows by turning everything inside out, as in Figure~\ref{inversion}.
\begin{figure}[htb]
\begin{center}
\includegraphics[width=4in]
					{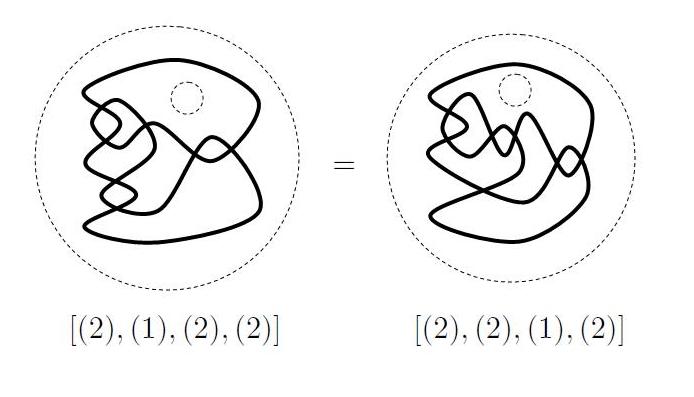}
\caption{These two knot shadows are essentially equivalent.  One is obtained from the the other by turning the diagram inside out,
exchanging the inner and outer dotted circles.}
\label{inversion}
\end{center}
\end{figure}
This works
because the diagram can be thought of as living on the sphere: note that the operation shown in Figure~\ref{knots-on-a-sphere} has no effect on a knot.
\begin{figure}[htb]
\begin{center}
\includegraphics[width=4in]
					{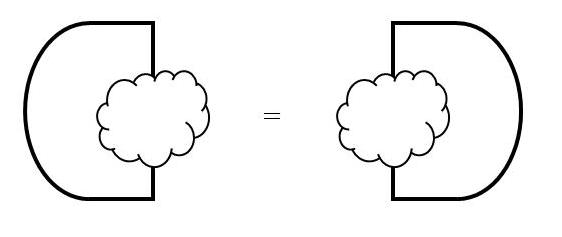}
\caption{Moving a loop from one side of the knot to the other has no effect on the knot.  So we might as well think
of knot diagrams as living on the sphere.}
\label{knots-on-a-sphere}
\end{center}
\end{figure}
\end{proof}

Similarly, we also have
\begin{lemma}
\begin{equation} \left[(0),(a_1 + 1),(a_2),\ldots,(a_n)\right] \stackrel{1}{\rightarrow} \left[(0),(a_1),(a_2),\ldots,(a_n)\right] \label{unwindl}\end{equation}
\begin{equation} \left[(a_1),\ldots,(a_{n-1}),(a_n + 1),0\right] \stackrel{1}{\rightarrow} \left[(a_1),\ldots,(a_{n-1}),(a_n),0\right] \label{unwindr}\end{equation}
\begin{equation} \left[\ldots,(a_i + 2),\ldots\right] \stackrel{2}{\rightarrow} \left[\ldots,(a_i),\ldots\right]\label{twoloss}\end{equation}
\end{lemma}
The proof is left as an exercise to the reader.

\begin{lemma}\label{rational-to-unknot}
If $T$ is a rational knot shadow, then $T \stackrel{*}{\Rightarrow} [~]$.
\end{lemma}
\begin{proof}
Let $T = \left[(a_1),\ldots,(a_n)\right]$ be a minimal counterexample.
Then $T$ cannot be reduced by any of the rules specified above.  Since any $a_i \ge 2$ can be reduced
by (\ref{twoloss}), all $a_i < 2$.  If $n = 0$, then $T = \left[~\right]$ which is
the unknot.  So $n$ is greater than 0, and $a_0$ is either $0$ or $1$.

\begin{itemize}
\item
If $a_0 = 0$ and $n > 1$, then either $a_1$ can be decreased
by $1$ using (\ref{unwindl}), or $a_0$ and $a_1$ can be stripped off via (\ref{zlossl}), contradicting minimality.
\item
If $a_0 = 0$ and $n = 1$, then $T = \left[(0)\right]$,
which is easily seen to be a two-component link, not a knot.
\item
If $a_0 = 1$ and $n > 1$, then $T$ reduces to
$\left[(a_2 + 1),\ldots,(a_n)\right]$ by (\ref{onecombl}), contradicting minimality.
\item
If $a_0 = 1$ and $n = 1$, then $T$ is $\left[(1)\right]$ which clearly reduces
to the unknot via a pseudo Reidemeister I move, a contradiction.
\end{itemize}
So in all four cases we have a contradiction.
\end{proof}

It then follows by Corollary~\ref{redux} that no rational
shadow has outcome K.  That is, no rational shadow exists which is a win for the Knotter no matter who goes first.

\section{Odd-Even Shadows}\label{sec:oddeven}
\begin{definition}
An \emph{odd-even shadow} is a shadow of the form \[\left[(a_1),(a_2),\ldots,(a_n)\right],\]
where exactly one of $a_1$ and $a_n$ is odd, and all other
$a_i$ are even.
\end{definition}
Note that these all have an odd number of crossings.
It is straightforward to verify from (\ref{onecombl}-\ref{twoloss})
that every odd-even shadow reduces by pseudo Reidemeister moves to the unknot.
In particular, by repeated applications of (\ref{twoloss}), we reduce to either
$[(0),\ldots,(0),(1)]$ or $[(1),(0),\ldots,(0)]$.  Then by applying (\ref{zlossl}) or
(\ref{zlossr}), we reach one of the following:
\[ [(1)], [(0),(1)], [(1),(0)].\]
Then all of these are equivalent to $[(1)]$ by (\ref{onecombl}) or (\ref{onecombr}).
So since every odd-even shadow reduces to the unknot, every odd-even shadow is an actual knot shadow, not a two-component link shadow.
Thus any odd-even shadow can be used as a position in the knotting-unknotting game.

\begin{lemma}\label{oddevenlemma}
If $T = [(a_1),(a_2),\ldots,(a_n)]$ is an odd-even shadow, then every option $T'$ of $T$ is either a fully resolved unknot,
or has an option which is equivalent to an odd-even shadow.
\end{lemma}
In other words, if the Knotter makes any move in an odd-even shadow, then either he has ended the game with a losing move,
or the Unknotter has a replying move which returns the position to an odd-even shadow.
\begin{proof}
By Equation~\ref{invert}, we can assume without loss of generality that $a_1$ is odd and the other $a_i$ are even.
If $a_1 = 1$, then by Equation~(\ref{onecombl}) we can choose a smaller representation, unless $T$ is $[(1)]$.  So we can assume
that either $T = [(1)]$, or $a_1 > 1$.  In the first case, the sole option of $[(1)]$ is the unknot.  In the other case,
every nonzero $a_i$ is at least 2, so any move in any twist can be cancelled with a move in the same twist, as in Figure~\ref{undoing-move}
\begin{figure}[htb]
\begin{center}
\includegraphics[width=5in]
					{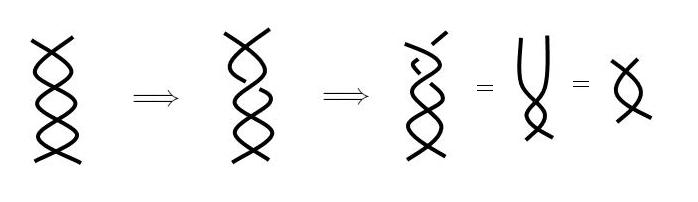}
\caption{The Unknotter responds to a twisting move by the Knotter with a cancelling twist in the opposite direction.}
\label{undoing-move}
\end{center}
\end{figure}
\end{proof}

\begin{theorem}\label{oddeven}
If $T$ is an odd-even shadow, then the normalized outcome of $T$ is $(U,U)$.
\end{theorem}
\begin{proof}
Equivalently (by Figure~\ref{diamond}), we need to show that $X(T) = Y(T) = 1$.
By Theorem~\ref{xyinequalities} and the fact that $T$ reduces to the unknot, $X(T) = 1$.  Because $T$ is odd,
$T = T^1$, so showing that $Y(T) = 1$ is the same as showing that
$T$ is U2.
In other words, we need to show that if the Knotter goes first and Unknotter goes second,
then the Unknotter wins under perfect play.

The Unknotter wins by the following strategy: in response to any move that does not end the game, she replies with a move to another
odd-even shadow.  Eventually the game ends with the Knotter moving to the unknot.  This strategy works by Lemma~\ref{oddevenlemma}.
\end{proof}
In Section~\ref{sec:sum} below, we will show that any connected sum of odd-even shadows also has normalized outcome $(U,U)$.
This will follow by showing that $T\#*$ is a zero game for every odd-even shadow $T$.  But first, we complete the classification
of rational shadows in the following section.

\section{The remaining cases}\label{sec:remainder}
For rational shadows which are not odd-even shadows, our strategy is to combine Corollary~\ref{21preserve} with the following
list of shadows, shown in Figure~\ref{the-irreducibles}:
\begin{lemma}\label{ohnoproof}
The following shadows have normalized outcome $(2,1)$:
\[ \left[(3),(1),(3)\right], \left[(2),(1),(2),(2)\right],
\left[(2),(2),(1),(2)\right], \left[(2),(1),(1),(2)\right],\]\[
\left[(2),(2),(1),(2),(2)\right], \left[(2),(2)\right]\]
\end{lemma}
\begin{proof}
Let $G$ be one of the specified shadows.  We already know that $X(G) = 1$, but we need to show (by Figure~\ref{diamond}) that $Y(G) = 3$.
By Definition~\ref{xydef} this amounts to showing that the even projection $G^0$ is K2.  This was proven by a computer
verification, making heavy use of Theorem~\ref{fundrat}.  See Appendix~\ref{sec:py} for details.
\end{proof}

\begin{figure}[htb]
\begin{center}
\includegraphics[width=4.5in]
					{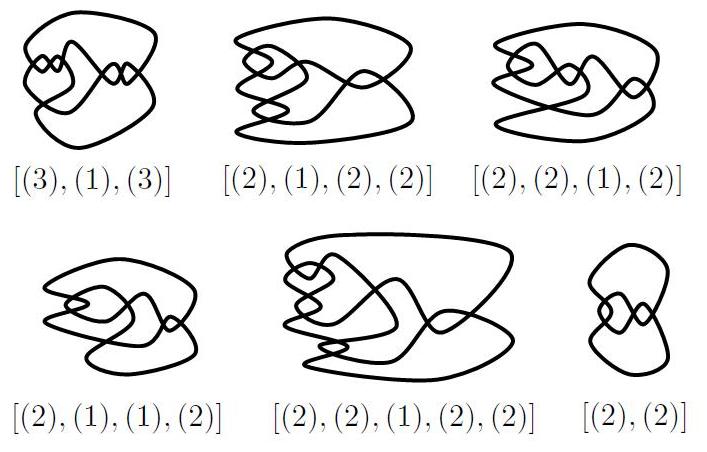}
\caption{The shadows of Lemma~\ref{ohnoproof}.}
\label{the-irreducibles}
\end{center}
\end{figure}

\begin{lemma}\label{minimal}
If $T = \left[(a_1),\ldots,(a_n)\right]$ is a rational knot shadow
with at least one crossing, then either $T \stackrel{1}{\Rightarrow} O$ for some odd-even shadow $O$, or
$T \stackrel{*}{\Rightarrow} A$, where $A$ is equivalent to one of the six shadows in Lemma~\ref{ohnoproof}.
\end{lemma}
\begin{proof}
Let $T$ be a counterexample, minimizing $a_1 + \cdots + a_n + n$.  If $a_1$ or $a_n = 0$, then by Equation (\ref{unwindl}) or (\ref{unwindr}),
$T \stackrel{1}{\to} T'$ for some smaller rational shadow $T'$.  Then by choice of $T'$, either $T' \stackrel{1}{\Rightarrow} O$ for some
odd-even shadow $O$, or $T \stackrel{*}{\Rightarrow} A$, where $A$ is one of the six shadows in Lemma~\ref{ohnoproof}.  By transitivity,
the same is true of $T$, so $T$ is not a counterexample.

By the same logic, if $a_i = 0$ for any $1 < i < n$, then Equation~(\ref{zloss}) shows that $T = T'$ for some $T'$ with smaller $n$,
and similarly we have a contradiction.  Thus we can assume that every $a_i > 0$.  The same argument applied to Equations~(\ref{onecombl}) and (\ref{onecombr})
implies that if $a_1 = 1$ or $a_n = 1$, then $n = 1$ and $T$ is the odd-even shadow $[(1)]$, a contradiction.  So we can also assume that $a_1$ and $a_n$
are at least 2.

If all of the $a_i$ are even, then by repeatedly applying
(\ref{twoloss}) and (\ref{zloss}),
we can reduce $T$ down to either $\left[(2),(2)\right]$ or $\left[(2)\right]$.  But the second of these
is easily seen to be a two-component link, so $T \stackrel{*}{\Rightarrow} \left[(2),(2)\right]$ and we are done.

Otherwise, at least one of the $a_i$ is odd.  Then since the $a_i$ are positive, and $a_1, a_n > 1$,
the following lemma implies that $T$ reduces by pseudo Reidemeister moves
to one of the six shadows of Lemma~\ref{ohnoproof}.
\end{proof}

\begin{lemma}\label{subcase}
If $T = [(a_1),\ldots,(a_n)]$ is a knot shadow such that every $a_i > 0$, $a_1 > 1$, $a_n > 1$, and at least one $i$ has
$1 < i < n$ and $a_i$ odd, then $T \stackrel{*}{\Rightarrow} A$, where $A$ is equivalent to one of the six shadows in Lemma~\ref{ohnoproof}.
\end{lemma}
\begin{proof}
Again, let $T$ be a counterexample, minimizing $a_1 + \cdots + a_n$.  If $a_j > 2$ for some $1 < j < n$, then we have
\[ T \stackrel{2}{\to} T' = [(a_1),\ldots,(a_{j-1}),(a_j - 2),(a_{j+1}),\ldots,(a_n)]\]
by Equation~\ref{twoloss}.  But then $T'$ satisfies all the same hypotheses as $T$, and is smaller, so by choice of $T$, we have
\[ T \stackrel{2}{\to} T' \stackrel{*}{\Rightarrow} A,\]
for some $A$.  Then $T \stackrel{*}{\Rightarrow} A$, a contradiction.
It follows that $a_j > 2$ for all $1 < j < n$.  The same reasoning shows that $a_1$ and $a_n$ are at most
3.

So at this point we can assume that:
\begin{itemize}
\item If $1 < j < n$, then $a_j$ is 1 or 2.
\item If $j = 1$ or $n$, then $a_j$ is 2 or 3.
\item There is some $i$ with $1 < i < n$ and $a_i$ odd (that is $a_i = 1$).
\item There is no way to reduce $T$ by pseudo Reidemeister moves to another pseudodiagram satisfying the same hypotheses as $T$.
\end{itemize}

Suppose that $a_1 = 3$.  Take the greatest possible $i$ such that $a_i = 1$.  If $i > 2$, then we can reduce $a_1$ by two using (\ref{twoloss}),
and combine it into $a_2$ by (\ref{onecombl}) to yield a smaller counterexample, contradicting minimality.  So $i = 2$,
and therefore $a_2 = 1$ and $a_j \ne 1$ for $j > 2$.  Thus, if a sequence begins with $3$, the next
number must be $1$, and the $1$ must be unique.  For example, the sequence $\left[(3),(1),(1),(3)\right]$
can be reduced to $\left[(1),(1),(1),(3)\right]$ and thence to $\left[(2),(1),(3)\right]$, contradicting minimality.

Next suppose that $a_1 = 2$.  Again take the greatest possible $i$ such that $a_i = 1$.  Suppose for the sake of contradiction
that $i > 4$.  We can reduce
$T$ further by decreasing $a_1$ by two, using (\ref{twoloss}), then clear $a_2$ by applying (\ref{unwindl}) enough times.
This makes $a_1$ and $a_2$ both zero, after which we can remove both by (\ref{zlossl}), yielding a smaller $T$.  A further application
of (\ref{onecombl}) may be necessary to remove an initial 1.  As long as $i > 4$, the resulting shadow will be a smaller counterexample,
contradicting minimality of $T$.  Moreover, if $a_3 = 2$, then the
final application of (\ref{onecombl}) is unnecessary, so there is a contradiction if $i > 3$.  In other words, if $a_1 = 2$,
then no $1$'s can occur beyond $a_4$, and if $a_3 = 2$, then no $1$'s can occur beyond $a_3$.

Therefore, what precedes any $a_i = 1$ must be one of the following:
\begin{itemize}
\item $(3)$
\item $(2)$
\item $(2)(2)$
\item $(2)(1)$
\item $(2)(2)(1)$
\item $(2)(1)(1)$
\end{itemize}
and only the first three of these can precede the first $(1)$.
By symmetry, the same sequences reversed must follow any $(1)$ in sequence.  Then $T$ must be one of the following combinations:
\begin{itemize}
\item $\left[(3),(1),(3)\right]$
\item $\left[(3),(1),(2)\right]$ and its reverse
\item $\left[(3),(1),(2),(2)\right]$ and its reverse
\item Not $\left[(3),(1),(1),(2)\right]$ because more than just $(3)$ precedes the second $(1)$.
\item $\left[(2),(1),(2)\right]$
\item $\left[(2),(1),(2),(2)\right]$ and its reverse
\item $\left[(2),(1),(1),(2)\right]$
\item $\left[(2),(1),(1),(2),(2)\right]$ and its reverse
\item $\left[(2),(1),(1),(1),(2)\right]$
\item $\left[(2),(2),(1),(2),(2)\right]$
\item $\left[(2),(2),(1),(1),(2),(2)\right]$
\item Not $\left[(2),(2),(1),(1),(1),(2)\right]$ because too much precedes the last $(1)$.
\end{itemize}
So either $T$ is one of the combinations in Lemma \ref{ohnoproof} or one of the following
happens:
\begin{itemize}
\item $\left[(3),(1),(2)\right]$ reduces by Equation (\ref{twoloss}) to $\left[(1),(1)
,(2)\right] = \left[(2),(2)\right]$, one of the shadows of  Lemma~\ref{ohnoproof}.  So does its reverse.
\item $\left[(3),(1),(2),(2)\right]$ reduces by two pseudo Reidemeister II moves to
\[\left[(3),(1),(0),(0)\right] = \left[(3),(1)\right] = \left[(4)\right]\] which is a two-component link shadow,
not a knot shadow.  Nor is its reverse.
\item $\left[(2),(1),(2)\right]$ reduces by a pseudo Reidemeister II move to $\left[(0),(1),(2)\right]$,
which in turn reduces by a pseudo Reidemeister I move to $\left[(0),(0),(2)\right] = \left[(2)\right]$
which is a two-component link, not a knot.
\item $\left[(2),(1),(1),(2),(2)\right]$ reduces by pseudo Reidemeister moves to
\[\left[(2),(1),(1),(0),(2)\right] = \left[(2),(1),(3)\right]\] which as noted above reduces to $\left[(2),(2)\right]$.
\item $\left[(2),(1),(1),(1),(2)\right]$ likewise reduces by a pseudo Reidemeister II move and
a I move to \[\left[(0),(0),(1),(1),(2)\right] = \left[(1),(1),(2)\right] = \left[(2),(2)\right],\]
one of the shadows of Lemma~\ref{ohnoproof}.
\item $\left[(2),(2),(1),(1),(2),(2)\right]$ reduces by a pseudo Reidemeister II move to
\[\left[(2),(0),(1),(1),(2),(2)\right] = \left[(3),(1),(2),(2)\right],\] which as noted above is a
two-component link shadow, not a knot shadow.
\end{itemize}
In other words, either $T$ is not a knot shadow, $T$ is one
of the shadows of Lemma~\ref{ohnoproof}, or $T$ reduces by pseudo Reidemeister moves to one of these shadows, specifically
$[(2),(2)]$.
The first case contradicts the assumption that $T$ is a knot shadow, and the other cases contradict
the assumption that $T$ is a counterexample.
\end{proof}
In summary then, every $T$ that does not reduce by pseudo Reidemeister I moves to
an odd-even shadow reduces down to a finite set of minimal cases.  Each of these minimal cases
is either reducible to one of the six shadows in Lemma \ref{ohnoproof}, or is not actually
a knot.

Consequently, this gives a rule for determining the outcome of a rational shadow:
\begin{theorem}\label{mainresult}
Let $T$ be a rational shadow.  If $T \stackrel{1}{\Rightarrow} O$ for some odd-even shadow $O$,
then $T$ is a win for the Unknotter, no matter which player goes first.  Otherwise, if $T$ is even
then $T$ is a win for the second player, and if $T$ is odd then $T$ is a win for the first player.
\end{theorem}
\begin{proof}
If $T \stackrel{1}{\Rightarrow} O$, then by Theorem~\ref{xyinequalities}, $X(T) = X(O)$
and $Y(T) = Y(O)$, so $T$ and $O$ have the same normalized outcome.  But by Theorem~\ref{oddeven},
$O$ has normalized outcome $(U,U)$.  So $T$ has normalized and extended outcomes $(U,U)$, and therefore its outcome
class is $U$.

Otherwise, by Lemma~\ref{minimal}, $T \stackrel{2}{\Rightarrow} A$ for $A$ equivalent to one of the six shadows
of Lemma~\ref{ohnoproof}.  Lemma~\ref{ohnoproof} says that $A$ has normalized outcome $(2,1)$, so by Corollary~\ref{21preserve},
the normalized outcome of $T$ must also be $(2,1)$.  Then if $T$ is even,
$T$ has extended outcome $(2,1)$, and if $T$ is odd, then $T$ has extended outcome $(1,2)$.  So the outcome of $T$
is 2 if $T$ is even, and 1 if $T$ is odd.
\end{proof}

Note that this tells us who wins rational \emph{shadows}, in which every crossing is unresolved.  It is still an open question
to determine who wins rational \emph{pseudodiagrams} in general.

As a special case of Theorem~\ref{mainresult}, we can rederive a result of~\cite{KnotGames}.
If $n \ge 2$ and $T = [(a_1),(a_2),\ldots,(a_n)]$ is a rational knot shadow
for which every $a_i$ is even, then $T$ cannot reduce by pseudo Reidemeister I moves at all, and is not an odd-even shadow.
So by Theorem~\ref{mainresult}, it has outcome class 2, because it has even parity.  This was Theorem 2 of \cite{KnotGames}.

\section{Sums of rational shadows}\label{sec:sum}

We now proceed to determine the outcome of all sums of rational knot shadows.  The key trick is to realize that
odd-shadows are essentially zero games.

\begin{definition}
A \emph{one-even} pseudodiagram is a pseudodiagram of the form \[ [\pm1 (a_1),(a_2),\ldots,(a_n)]\]
or \[ [(a_1),(a_2),\ldots,\pm1(a_n)] \] for $n \ge 1$ and even $a_i$.
\end{definition}
It is clear by the argument of Figure~\ref{undoing-move} that every option of a one-even pseudodiagram
has an option which is equivalent to a one-even pseudodiagram.  Moreover, we have the following
\begin{lemma}
Every one-even pseudodiagram is either fully resolved as an unknot, or has an option which is equivalent to an odd-even shadow.
\end{lemma}
\begin{proof}
First consider the case where $T$ is of the form $[\pm1(a_1),(a_2),\ldots,(a_n)]$.  If $a_1$ is zero and $n > 1$, then for
reasons analogous to Equation~\ref{onecombl}, we see that
\[ T = [\pm1,(a_2),\ldots,(a_n)] = [\pm1(a_2),\ldots,(a_n)].\]
Repeating this if necessary, we can assume without loss of generality that either $a_1 > 0$ or $n = 1$.  If $a_1 > 0$,
then an appropriate move in $(a_1)$ produces
\[ [0(a_1 - 1),(a_2),\ldots,(a_n)],\]
an odd-even shadow.  If $a_1 = 0$, then $n = 1$, and $T$ is $[1(0)] = [1]$, which is a fully resolved unknot.

The case where $T$ is of the form $[(a_1),(a_2),\ldots,1(a_n)]$ is handled similarly.
\end{proof}
From this we can produce a family of zero games:
\begin{lemma}
If $T$ is a one-even pseudodiagram, then $T$ is a zero game.  Also, if $S$ is an odd-even shadow, then $S\#*$ is a zero game.
\end{lemma}
\begin{proof}
We first show inductively that every one-even pseudodiagram is a zero game.  Let $T$ be a one-even pseudodiagram.
Then every option of $T$ has an option which is a one-even pseudodiagram, so by induction, every option of $T$
has an option which is a zero game.  Also, $T$ is either a fully resolved unknot (in which case it is definitely U1),
or $T$ has an option which is an odd-even shadow.
In this case, it follows that $T$ is U1 because the Unknotter can move from $T$ to an odd-even shadow, which is a win for the Unknotter
by Theorem~\ref{oddeven}.  So $T$ is a zero game.

Next we show inductively that if $S$ is an odd-even shadow, then $S\#*$ is a zero game.  First of all, $S\#* = S^0$ is U1 by Theorem~\ref{oddeven}.
So it remains to show that we can respond to any move from $S\#*$ with a move producing a zero game.

First of all suppose that our opponent moved in $*$, turning it into the unknot.  Then the total game becomes equivalent to $S$.  By moving
in the odd twist of $S$, we can produce a one-even pseudodiagram.  We just showed that these are all zero games.

Next suppose that our opponent moved in $S$, producing a total position of the form $S'\#*$, where $S'$ is some option of $S$.  Then
by Lemma~\ref{oddevenlemma}, either $S'$ is a fully resolved unknot, or $S'$ has an option $S''$ which is another odd-even shadow.
If $S'$ is a fully resolved unknot, then $S'\#*$ is equivalent to $*$, from which we can move to a fully resolved unknot, which is a zero game.
Otherwise, we can move to $S''\#*$, which is a zero game by induction.
\end{proof}

\begin{lemma}\label{odd-even-sum-normalized}
A sum of odd-even shadows has normalized outcome $(U,U)$.
\end{lemma}
\begin{proof}
Let $\Sigma = S_1 \# S_2 \# \cdots \# S_n$ be a sum of odd-even shadows.  Consider the positions
\[ S_1\#*\#S_2\#*\#\cdots\#S_n\#*\]
and
\[ S_1\#*\#S_2\#*\#\cdots\#S_n\]
By repeated applications of Corollary~\ref{r1}, these have the same outcomes as $\Sigma^0$ and $\Sigma^1$ respectively.
But since $S_i\#*$ is a zero game for every $i$, by repeated applications of Theorem~\ref{zerogames}, these also have the same outcomes
as $S_n\#*$ and $S_n$ respectively.  But both $S_n\#*$ and $S_n$ have outcome U by Theorem~\ref{oddeven}, so both $\Sigma^0$
and $\Sigma^1$ do too.
\end{proof}

Using these results we generalize Theorem~\ref{mainresult}
\begin{theorem}
Let $T_1, \ldots, T_n$ be rational shadows, and let $S = T_1\#T_2\# \cdots \#T_n$ be their connected sum.  If for every $T_n$ there exists an odd-even shadow $O_n$ with
$T_n \stackrel{1}{\Rightarrow} O_n$, then $S$ is a win for the Unknotter, no matter which player goes first.
Otherwise,
\begin{itemize}
\item If $S$ has an even number of crossings, then $S$ is a win for whichever player goes second.
\item If $S$ has an odd number of crossings, then $S$ is a win for whichever player goes first.
\end{itemize}
\end{theorem}
\begin{proof}
In the first case, there is an odd-even shadow $O_n$ for every $n$, such that $T_n \stackrel{1}{\Rightarrow} O_n$.  Then by repeated applications
of Remark~\ref{rem},
\[ S = T_1 \# T_2 \# \cdots \# T_n \stackrel{1}{\Rightarrow} O_1 \# O_2 \# \cdots \# O_n.\]
Then by Theorem~\ref{xyinequalities}, $S$ has the same normalized outcome as $O_1 \# O_2 \# \cdots \# O_n$.
But this normalized outcome is $(U,U)$ by Lemma~\ref{odd-even-sum-normalized}.  So the normalized (and extended) outcome
of $S$ is $(U,U)$.  Therefore $S$ has outcome $U$ - it is a win for the Unknotter.

In the second case, by Lemma~\ref{minimal} there is some $n$ such that $T_n \stackrel{*}{\Rightarrow} A$ where $A$ is one of the knots
of Lemma~\ref{ohnoproof}.  But then by Lemma~\ref{rational-to-unknot} and repeated applications of Remark~\ref{rem},
\[ S = T_1\#T_2\# \cdots \# T_n \stackrel{*}{\Rightarrow} [~]\# [~] \# \cdots \# [~] \# A \# [~] \# \cdots \# [~] = A\]
where $[~]$ is the unknot as usual.

So since $A$ has normalized outcome $(2,1)$, by Lemma~\ref{ohnoproof}, it follows by Corollary~\ref{21preserve} that $S$ also
has normalized outcome $(2,1)$.  Thus the extended outcome of $S$ is $(2,1)$ if $S$ is even, and $(1,2)$ otherwise.  Consequently
the outcome of $S$ is exactly what was stated above.
\end{proof}
\section{Conclusion}
In summary, we have determined which player wins the knotting-unknotting game, when the initial position is a sum of rational shadows.
Rational shadows fall into two classes.  One class consists of the trivial shadow $[~]$ and all rational shadows which reduce
to odd-even shadows by repeated applications of (\ref{onecombl}-\ref{zlossr}) and (\ref{unwindl}-\ref{unwindr}), while the other class consists
of all other rational shadows.  Then a sum of one or more rational shadows has normalized outcome $(2,1)$ if any of the summands is in the second class,
and normalized outcome $(U,U)$ otherwise.

While we have a complete rule for the case of rational shadows, little is known about the case of rational \emph{pseudodiagrams}.  This case would be amenable
to computer explorations because we have a precise test for the unknot in this case.  Corollary~\ref{r1} and Lemmas~\ref{r1bland}-\ref{r2} still apply in this
case and could be used to simplify calculations.

Another idea worth pursuing is the study of the combinatorial game theory of connected sums.  The basic idea is to find a function $f$ from pseudodiagrams
to some monoid $M$, with the properties that $f(S\#T) = f(S) + f(T)$ and the outcome of a pseudodiagram $S$ is determined by $f(S)$.
The function $f$ condenses all the relevant strategic information about a position.
For example, in the case of sums of rational shadows, we could let $f(S) = (\pi(S),\kappa(S))$, where $\pi(S)$ is the parity of $S$
and $\kappa(S)$ is 0 or 1 depending on whether the normalized outcome of $S$ is $(U,U)$ or $(2,1)$.
Then the outcome of $S$ is determined by $f(S)$, and $f(S\#T)$ is predictable
from $f(S)$ and $f(T)$.  We would like a general version of $f$ for all pseudodiagrams.

The canonical choice for $f$ is the map from pseudodiagrams to equivalence classes of pseudodiagrams modulo \emph{strategic equivalence}, defined as follows:
\begin{definition}
Two pseudodiagrams $S$ and $T$ are \emph{strategically equivalent} if for all pseudodiagrams $Q$, the sums $S\#Q$ and $T\#Q$ have the same outcome.
\end{definition}
The quotient space of pseudodiagrams modulo strategic equivalence has a monoid structure induced by $\#$.
By a complicated case-by-case analysis of a much larger class of games, one can show that there are at most 37 equivalence classes of pseudodiagrams.
However, the actual number seems to be much smaller, and further work is needed to determine the full account of this structure.

\section{Acknowledgments}
I would like to thank Allison Henrich who edited this paper and first introduced me to the knotting-unknotting game.  This research was
done during 2010 and 2011 in the University of Washington's Mathematics REU in inverse problems, which is run by James Morrow.
\appendix
\section{Python code}\label{sec:py}
I used the following straightforward brute-force python code:
\begin{verbatim}

def unknot(fraction):
    # We assume the first and last numbers are not irrelevant
    num = 0
    denom = 1
    for k in fraction:
        num += denom*k
        num, denom = denom, num
    # The last number modified is now denom
    # Since we're assuming that the last modification
    # was important...
    if(denom == (denom/2)*2):
        print "Actually a two-component link!", fraction
        return False
    else:
        return abs(denom) == 1



def opponent(whose):
    if(whose == "Ursula"):
        return "Lear"
    return "Ursula"

        
def recursiveEval(template, state, whose):
    if(state[0] == 0): #none remain
        unk = unknot(state[1])
        if(unk):
            return "Ursula"
        else:
            return "Lear"
    else:
        for i in range(len(template)):
            if(state[2][i] > 0):
                # do the move
                state[2][i] -= 1
                state[0] -= 1
                state[1][i] += 1
                winner = recursiveEval(template, state, opponent(whose))
                state[1][i] -= 1
                state[0] += 1
                state[2][i] += 1
                if(winner == whose):
                    return whose
                # do the other move
                state[2][i] -= 1
                state[0] -= 1
                state[1][i] += -1
                winner = recursiveEval(template, state, opponent(whose))
                state[1][i] -= -1
                state[0] += 1
                state[2][i] += 1
                if(winner == whose):
                    return whose
        return opponent(whose)


def evaluateKnot(template):
    state = [sum(template), [0]*len(template), template[:]]
    ursFirst = recursiveEval(template, state, "Ursula")
    print "If Ursula goes first, the winner is",ursFirst
    learFirst = recursiveEval(template, state, "Lear")
    print "If Lear goes first, the winner is",learFirst
\end{verbatim}
Then for example \texttt{evaluateKnot([2,2])} produces the output
\begin{verbatim}
If Ursula goes first, the winner is Lear
If Lear goes first, the winner is Ursula
\end{verbatim}

By equations (\ref{zlossl}) and (\ref{unwindl}), if $[(a_1),(a_2),\ldots,(a_n)]$ is any shadow then
\[ [(0),(1),(a_1),(a_2),\ldots,(a_n)] \stackrel{1}{\to} [(0),(0),(a_1),(a_2),\ldots,(a_n)] = [(a_1),(a_2),\ldots,(a_n)].\]
So we can find the extended outcome of $[(a_1),(a_2),\ldots,(a_n)]$ by checking the outcome of
$[(a_1),\ldots,(a_n)]$ and $[(0),(1),(a_1),\ldots,(a_n)]$.  Then to verify Lemma~\ref{ohnoproof},
we make the following calls, which produce the expected results:
\begin{verbatim}
>>> evaluateKnot([3,1,3])
If Ursula goes first, the winner is Ursula
If Lear goes first, the winner is Lear
>>> evaluateKnot([0,1,3,1,3])
If Ursula goes first, the winner is Lear
If Lear goes first, the winner is Ursula
>>> evaluateKnot([2,1,2,2])
If Ursula goes first, the winner is Ursula
If Lear goes first, the winner is Lear
>>> evaluateKnot([0,1,2,1,2,2])
If Ursula goes first, the winner is Lear
If Lear goes first, the winner is Ursula
>>> evaluateKnot([2,2,1,2])
If Ursula goes first, the winner is Ursula
If Lear goes first, the winner is Lear
>>> evaluateKnot([0,1,2,2,1,2])
If Ursula goes first, the winner is Lear
If Lear goes first, the winner is Ursula
>>> evaluateKnot([2,1,1,2])
If Ursula goes first, the winner is Lear
If Lear goes first, the winner is Ursula
>>> evaluateKnot([0,1,2,1,1,2])
If Ursula goes first, the winner is Ursula
If Lear goes first, the winner is Lear
>>> evaluateKnot([2,2,1,2,2])
If Ursula goes first, the winner is Ursula
If Lear goes first, the winner is Lear
>>> evaluateKnot([0,1,2,2,1,2,2])
If Ursula goes first, the winner is Lear
If Lear goes first, the winner is Ursula
>>> evaluateKnot([2,2])
If Ursula goes first, the winner is Lear
If Lear goes first, the winner is Ursula
>>> evaluateKnot([0,1,2,2])
If Ursula goes first, the winner is Ursula
If Lear goes first, the winner is Lear
\end{verbatim}
All of these evaluated in less than a second or two, except the call
\begin{quote}
\texttt{evaluateKnot([0,1,2,2,1,2,2])}\end{quote} which took under 10 seconds.

\bibliographystyle{plain}
\bibliography{Master}{}

\end{document}